\documentclass[a4paper,11pt,centertags]{amsart}
\usepackage[latin1]{inputenc}
\usepackage[english]{babel}
\usepackage{cite}
\usepackage{verbatim}
\usepackage{color}
\usepackage{hyperref}
\usepackage{extarrows}
\usepackage{amsmath,amsthm}

\usepackage{comment}

\usepackage[top=3.2cm, bottom=3.2cm, left=2.8cm,right=2.8cm]{geometry}

\newcommand{\diesis}{^\#}
\newtheorem{theo}{Theorem}[section]

\newtheorem{prop}{Proposition}[section]

\DeclareMathOperator{\dint}{\displaystyle\int}
\DeclareMathOperator\erf{erf}

\theoremstyle{definition}
\newtheorem{definiz}{Definition}[section]
\newtheorem{rem}{Remark}[section]

\numberwithin{equation}{section}

\newcommand{\R}{\mathbb R}

\newcommand{\f}{\mathcal F}
\newcommand{\de}{\partial}

\newcommand{\ds}{\displaystyle}
\DeclareMathOperator{\divergenza}{div}

\begin{document}
\title[The Faber-Krahn inequality for the Hermite operator with Robin BCS]
{The Faber-Krahn inequality for the Hermite operator with Robin boundary condition}

\author[F. Chiacchio]{Francesco Chiacchio}
\address{
Francesco Chiacchio \\
Universit\`a degli Studi di Napoli Federico II\\
Dipartimento di Ma\-te\-ma\-ti\-ca e Applicazioni ``R. Caccioppoli''\\
Complesso di Monte Sant'Angelo, Via Cintia,
80126 Napoli, Italia.
}
\email{francesco.chiacchio@unina.it}

\author[N. Gavitone]{Nunzia Gavitone}
 \address{
Nunzia Gavitone \\
Universit\`a degli Studi di Napoli Federico II\\
Dipartimento di Ma\-te\-ma\-ti\-ca e Applicazioni ``R. Caccioppoli''\\
Complesso di Monte Sant'Angelo, Via Cintia,
80126 Napoli, Italia.
}
\email{nunzia.gavitone@unina.it}

\date{\today}
\maketitle
\begin{abstract}
In this paper we prove a Faber-Krahn type inequality for the first eigenvalue of the Hermite operator with Robin boundary condition. We prove that the optimal set is an half-space and we also address the equality case in such inequality. 
\end{abstract}

\vspace{.5cm}


\vspace{.5cm}
\section{Introduction}

This paper deals with the eigenvalue problem for the Hermite operator with  Robin boundary condition.
Let us denote by 
\[
d\gamma_N =\phi_N(x) \,dx
\]
the normalized Gaussian measure in $\R^N$, where  its density, $\phi_N(x)$, is given by 
\begin{equation*}
\phi _{N}(x)=\frac{1}{\left( 2\pi \right) ^{\frac{N}{2}}}\exp \left( -\frac{
|x|^{2}}{2}\right) .
\end{equation*}
 Let  $\Omega $ be a sufficiently smooth domain
of $\mathbb R^N$  
and let $\nu $ be the unit outer normal to 
$\partial \Omega $. We consider the following eigenvalue problem
\begin{equation}
 \label{PI}
\left\{ 
\begin{array}{ll}
-\text{div}\left( \phi _{N}(x)\nabla u(x))\right) =\lambda (\Omega )\phi
_{N}(x)u(x) & \text{in }\Omega  \\[0.2cm]
&  \\ 
\dfrac{\partial u}{\partial \nu }+\beta u=0 & \text{on }\partial \Omega ,
\end{array}
\right.  
\end{equation}
where $\beta >0$.
Our aim is to prove a Faber-Krahn
inequality for the first eigenvalue, $\lambda _{1}(\Omega )$, of problem \eqref{PI}.
More precisely,  
 we are interested in finding, in the class of the sets having prescribed
Gaussian measure, the one which minimizes $\lambda _{1}(\Omega )$.

The  Faber-Krahn inequality for the first Robin eigenvalue, for the classical Laplace operator, has been established in the planar case  in \cite{bos}. Subsequently, such a result has been generalized in any dimension in \cite{da06} and for the $p$-Laplace operator in \cite{bd10}. In all these cases the ball turns out to be the optimal set, when the Robin parameter $\beta$ is positive.

For $\beta<0$, the problem appears more delicate. 
Indeed   many phenomena can occur,
depending also on the dimension and the topological properties of the domain
(see, e.g., \cite{FK}, \cite{FNT}, \cite{AFK} and the references therein).

The difficulties of the problem when $\beta >0$ are twofold. From on hand the level sets of
the first eigenfunction are not closed, on the other, there is no
monotonicity of the first eigenvalue with respect to the inclusion of sets.
These facts prevent one from adapting the classical symmetrization methods, which
work for the Dirichlet boundary condition. In order to overcome these difficulties, in the papers quoted above, the authors used a   different approach. 
More precisely, in \cite{bos},  Bossel introduced a sort of ``desymmetrization" technique together with a representation formula for the first Robin eigenvalue. Concerning problem \eqref{PI}, some further issues arise since, in general, the domain $\Omega $ is unbounded and the first eigenfunction is not in $L^{\infty }(\Omega ). $
These circumstances do not allow  to apply, in a straightforward way, the Bossel's arguments. Our analysis requires a detailed study of problem \eqref{eq:2} in one dimension (see Section 3) and a suitable modification of the proof of the representation formula (see Section 4). One crucial point, in the study of the one-dimensional problem, is the  asymptotic behavior of the first eigenfunction, since it  allows  to prove, among other things, its $\log$-concavity which, in turn,  enables to prove the monotonicity property, with respect the inclusion of half lines, of the first eigenvalue.

Let us spend a few words about some motivations of the present note.
First of all, the Hermite operator, as well known,  enter in the description of the harmonic oscillator in quantum mechanics (see, e.g., \cite{BF} and the references therein). 
Moreover the interest in the Hermite operator comes from the fact that
Gaussian measure in $\mathbb R^N$ can be obtained as a limit, as $k$ goes to infinity, of the 
normalized surface measures on
$\mathbb S_{\sqrt{k}}^{k+N+1}$, 
the sphere in $\mathbb R^{k+N+2}$ of radius $\sqrt{k}$ (a process known in literature as ``Poincar\`{e} limit").

Note also that the Robin problem is often regarded as interpolating 
between the Dirichlet  ($\beta = 0$)  and Neumann ($\beta = +\infty $) cases. 

As far as the Hermite operator is considered and $\Omega$ varies in the class of domains with fixed Gaussian measure, it is known (see \cite{Eh} and \cite{BCF}) that 
the set which minimizes the first Dirichlet eigenvalue is given by an half-space. Hence it coincides with the isoperimetric set in the Gaussian Isoperimetric inequality (see, e.g., \cite{ST}, \cite{Bo}, \cite{CK}, \cite{CFMP} and the references therein).
On the other hand, when Neumann boundary condition are imposed, the situation
is quite different. As well known,  the problem of minimizing the first non-trivial
Neumann eigenvalue, $\mu_{1}(\Omega )$,
is meaningless and, hence, one tries to maximize it, instead. In \cite{ChdB} it
has been proved that among all smooth domains with prescribed Gaussian measure,
symmetric with respect to the origin the ball centered at the origin
maximizes  $\mu_{1}(\Omega )$.
Furthermore in \cite{ChdB} it is shown that, even removing such a topological
assumption, the half spaces  do not maximize it. 

This phenomenon is quite surprising since in the Euclidean case the ball
minimizes the first Dirichlet eigenvalue (the classical Faber-Krahn
inequality) and, at same time, it maximizes the first nontrivial Neumann
eigenvalue (the classical Szeg\" o-Weiberger inequality).

So it is natural  to investigate  which is the optimal set 
in the ``Gaussian-Faber-Krahn inequality"
when  Robin boundary condition are imposed. 

We finally point out that in  \cite{cgnt} the authors prove an isoperimetric inequality for the Robin torsional rigidity  related to the Hermite operator.

Our main result, Theorem \ref{main} below, requires
the validity of some functional embedding Theorems. In Section 2 (see definition \ref{omega})
we introduce and describe  the  family  $\mathcal G$ of those 
domains of $\mathbb R^N$ for which such results hold true.

\begin{theo}
\label{main}
Let $\Omega \in  \mathcal G  $, with $\mathcal G$ as in Definition \ref{omega}. Then
\begin{equation}
\label{fktesi}
 \lambda_1(\Omega^{\diesis})\le \lambda_1(\Omega),
\end{equation}
where $\Omega^{\diesis}=S_{\sigma^{\diesis}}$ is any half-space  such that $\gamma_N(\Omega^{\diesis})=\gamma_N(\Omega)$. Moreover,  equality holds

\vspace{.1cm}
\noindent in \eqref{fktesi}, if and only if, $\Omega$ is an half-space, modulo a rotation about the origin.  
\end{theo}
The paper is organized as follows. Section 2 contains the isoperimetric inequality with respect to the Gaussian measure, along with some preliminary results on the eigenvalue problem \eqref{PI}. In the third Section we analyze the problem in one dimension. As said before, among other things, we prove  the $\log$-concavity of any  positive first eigenfunction.
Note that in $\mathbb R^N$ such a property does not hold true, in general, even for the classical Laplace operator (see \cite{ACH}).
 Section 4   is devoted to the representation formula for the first eigenvalue of the problem \eqref{PI}. Finally, in the last Section, we prove our main result, Theorem \ref{main}.

\section{Preliminary results} 
 
\subsection{The Gaussian isoperimetric inequality}

Let $\phi_N (x)$ denotes the density of the normalized Gaussian measure in
 $\mathbb{R}^{N}\,$ i.e. 
\begin{equation*}
\phi _{N}(x)=\frac{1}{\left( 2\pi \right) ^{\frac{N}{2}}}\exp \left( -\frac{
|x|^{2}}{2}\right) .
\end{equation*}
One defines the Gaussian perimeter of any Lebesgue measurable set $A$ of $
\mathbb{R}^{N}$ as follows 
\begin{equation*}
P_{\phi_N }(A)=\left\{ 
\begin{array}{cc}
\displaystyle\int_{\partial A}\phi_N (x)\text{ }d\mathcal{H}^{N-1}(x) & \text{
if }\partial A\text{ is }(N-1)-\text{rectifiable} \\ 
&  \\ 
+\infty  & \text{otherwise,}
\end{array}
\right. 
\end{equation*}
where $d\mathcal H^{N-1}$ denotes the $(N-1)-$dimensional Hausdorff measure in $\R^N$.

While the Gaussian measure of $A$ is given by 
\begin{equation}
\gamma_N(A)=\int_{A}\phi_N (x)\,dx\in \left[ 0,1\right] .  \label{mis_pes}
\end{equation}
The celebrated Gaussian isoperimetric inequality (see \cite{ST}, \cite{Bo}
and \cite{Eh}) states that among all Lebesgue measurable sets in $\mathbb{R}
^{N},$ with prescribed Gaussian measure, the half-spaces 
minimize the Gaussian perimeter. Furthermore the isoperimetric set is
unique, clearly, up to a rotation with respect to the origin (see \cite{CK}
and \cite{CFMP}).

Let $a\in \mathbb{R}$, in the sequel we will use the following notation 
\begin{equation}
\label{ss}
S_{a}:=\{x=(x_{1},x_{2},\cdots ,x_{n})\in \mathbb{R}^{N}\text{ }\colon \text{
}x_{1}<a\}.
\end{equation}
If $\Omega \subset \mathbb{R}^{N}$ is a Lebesgue measurable set, its
Gaussian symmetrized, $\Omega ^{\diesis },$ is
\begin{equation}
\Omega ^{\diesis }=S_{\sigma ^{\diesis }}  \label{S}
\end{equation}
where $\sigma ^{\diesis }$ is such that 

\begin{equation}
\gamma_N( \Omega)=\gamma_N(\Omega ^{\diesis
})=\frac{1}{\sqrt{2\pi }}\int_{-\infty }^{\sigma ^{\diesis
}}\,\text{exp}\left( -\dfrac{t^{2}}{2}\right) dt=\frac{1}{2}+\frac{1}{2}\erf
\left( \frac{\sigma ^{\diesis }}{\sqrt{2}}\right),  \label{sigma}
\end{equation}
where $\erf$ stands for the standard error function.

The isoperimetric function in Gauss space, $g(s),$ is given by
\begin{equation*}
g:s\in \left[ 0,1\right] \rightarrow g(s)=\frac{1}{\sqrt{2\pi }}\exp \left[ -
\frac{\left( \erf^{-1}(2s-1)\right) ^{2}}{2}\right].
\end{equation*}
Note, indeed, that the Gaussian perimeter of any half-space of Gaussian
measure $s\in \left[ 0,1\right] $ is equal to $g(s).$ The isoperimetric
property of the half-spaces can finally be stated also in the following
analytic form.

\begin{theo}
\label{isop} If $\Omega \subset \mathbb{R}^{N}$ is any Lebesgue measurable
set it holds that

\begin{equation*}
P_{\phi_N }(\Omega )\geq P_{\phi_N }(\Omega ^{\diesis})
=
g \left(\gamma_N(
\Omega)\right),
\end{equation*}
where equality holds, if and only if, $\Omega $ is equivalent to an half-space.
\end{theo}

\bigskip 

\subsection{The Sobolev space \textbf{$H^1(\Omega,\phi_N)$}}

Let $\Omega \subset \mathbb R^n$ be an open connected set.
     We will denote by $L^{2}(\Omega ,\phi_N )$ the set of all real
measurable functions defined in $\Omega$, such that
\begin{equation*}
\left\Vert u\right\Vert _{L^{2}(\Omega ,\phi_N )}^{2}:=\int_{\Omega
}u^{2}(x)\phi_N (x)dx<+\infty .
\end{equation*}
For our future purposes we need also to introduce the following weighted
Sobolev space
\begin{equation*}
H^{1}(\Omega ,\phi_N ):=\left\{ u\in W_{\text{loc}}^{1,1}(\Omega
):(u,\left\vert \nabla u\right\vert )\in L^{2}(\Omega ,\phi_N )\times L^{2}(\Omega
,\phi_N )\right\} ,
\end{equation*}
endowed with the norm
\begin{equation*}
\left\Vert u\right\Vert _{H^{1}(\Omega ,\phi_N )}=\left\Vert u\right\Vert
_{L^{2}(\Omega ,\phi_N )}+\left\Vert \nabla u\right\Vert _{L^{2}(\Omega ,\phi_N )}.
\end{equation*}
In the sequel of the paper, we need to introduce the following family of sets..
\begin{definiz}
\label{omega}
A Lipschitz domain $\Omega$ of $\mathbb R^N$ is in $\mathcal G$ 
if  $\gamma_N(\Omega) \in (0,1)$
and  the following conditions are fulfilled:
\begin{itemize}
\item[(i)]  $H^{1}(\Omega ,\phi_N )$ is compactly embedded in 
$L^{2}(\Omega ,\phi_N )$. 
\item[(ii)]  The trace operator $T$
\begin{equation*}
T:u\in H^{1}(\Omega ,\phi_N )\rightarrow \left. u\right\vert _{\partial \Omega
} \in L^{2}(\partial \Omega ,\phi_N ),
\end{equation*}
is well defined;

\item[(iii)] The trace operator defined in the previous point is compact from $H^{1}(\Omega ,\phi_N )$ onto  $L^{2}(\partial \Omega ,\phi_N )$.

\end{itemize}
In (ii) and (iii) the functional space $L^{2}(\partial \Omega ,\phi_N )$ is endowed with the
norm
\begin{equation*}
\left\Vert u\right\Vert _{L^{2}(\partial \Omega ,\phi_N )}^{2}=\int_{\partial
\Omega }u^{2}(x)\phi_N (x)d\mathcal{H}^{N-1}(x).
\end{equation*}
\end{definiz}

\begin{rem}
Observe that $\mathcal G$ is non empty. Indeed it contains, at least, the following families of sets
\begin{itemize}
\item[(j)] All the bounded and Lipschitz domains of $\mathbb R^N$.

\item[(jj)] All the convex domains of $\mathbb R^N$, not necessarily bounded.

\item[(jjj)] All the Lipschitz domains of $\mathbb R^N$, not necessarily bounded,
 for which an extension Theorem holds true (see, e.g., \cite{FP}, \cite{HR}, 
\cite{BCKT} and \cite{Li}).

\end{itemize}
%

\end{rem}

%




The study of functional inequalities related to the Gaussian measure,
also because they are often generalizable  to infinite-dimensional spaces,
has given rise to a rich line of research, starting from the seminal paper by Gross (see 
\cite{Gross}). 
The related bibliography is very wide, we remind the interested reader to 
\cite{Adams}, \cite{ADPP}, \cite{BH}, \cite{B},  \cite{BCT},  \cite{C2008}, 
\cite{CMP},     \cite{PT} and the references therein.

\subsection{The eigenvalue problem}
\label{autval}
Let $\Omega \in \mathcal G$, with $\mathcal G$ as in Definition \ref{omega}.  We finally consider the  eigenvalue Robin boundary value problem
\begin{equation}
  \label{eq:2}
  \left\{
    \begin{array}{ll}
      -\divergenza\left(\phi_N(x)\nabla u(x))\right)=\lambda(\Omega) \phi_N(x) u(x)
      &\text{in } 
      \Omega,\\[.2cm] 
      \displaystyle\frac{\partial u}{\partial \nu} +\beta u=0
      &\text{on } \de\Omega. 
    \end{array}
  \right.
\end{equation}
where $\nu$ is the unit outer normal to $\partial\Omega$ and  $\beta>0$. 
A real number $\lambda(\Omega)$ and a function $u\in H^1(\Omega,\phi_N)$ are respectively called eigenvalue and corresponding eigenfunction of the problem  \eqref{eq:2}  if the following equality holds
  \begin{equation}
  \label{defsol}
  \int_{\Omega}\langle \nabla u, \nabla\psi\rangle\, \phi_N  \,dx +\beta \int_{\de \Omega} u\psi \phi_N\, \,d\mathcal H^{N-1}(x)=\lambda(\Omega)\int_{\Omega} u\,\psi\,\phi_N \,dx,\quad \psi\in H^{1}(\Omega,\phi_N).
  \end{equation}

By the definition of  $ \mathcal G$,  one can apply 
the standard 
theory on compact, self-adjoint operator
to problem  \eqref{eq:2}. Hence 
we can arrange the eigenvalues of \eqref{eq:2}
in a non-decreasing sequence 
 $\{   \lambda_n(\Omega)  \}_{n \in \mathbb N}   $,
 such that $\lim_{n \to \infty}\lambda_n(\Omega)=+\infty$.
Moreover, the smallest eigenvalue of \eqref{eq:2}, $\lambda_1(\Omega)$,
 has the following variational characterization
\begin{equation}\label{var.beta}
\lambda_1(\Omega)=\inf_{v \in H^1(\Omega,\phi_N)\setminus\{0\}}J[\beta,v],
\end{equation}
where 
  \begin{equation}
  J[\beta,v]=\frac{\ds\int_{\Omega} |\nabla v|^2 \phi_N\, dx+\beta\ds\int_{\de\Omega}v^2 \phi_N \,d\mathcal H^
{N-1}(x)}{\ds\int_{\Omega}v^2 \phi_N\, dx}.
  \end{equation}
Conversely, any function on which the functional  $J[\beta,v]$ achieves its minimum is an eigenfunction of  \eqref{eq:2}.
Furthermore the following result holds.
\begin{theo}
\label{esi}
Let $\Omega \in \mathcal G$, with $\mathcal G$ as in Definition \ref{omega} and let  $\lambda_1(\Omega)$ be the first eigenvalue of \eqref{eq:2}. Then  the corresponding eigenfunctions are smooth,  have  constant sign in $\Omega$ and $\lambda_1(\Omega)$ is simple, that is the first eigenfunction is unique up to a multiplicative constant. 
\end{theo}
\begin{proof}
By classical regularity results the first eigenfunctions are smooth in $\Omega$. Let  $u $ be  an eigenfunction corresponding to  $\lambda_1(\Omega)$, then  both  $u $ and $|u|$ minimize  \eqref{var.beta},
 since $J(\beta, u)=J(\beta, |u|)$. Therefore  $|u|$ is a first eigenfunction too. By Harnack's inequality we deduce  that $|u|$  can not vanishes inside 
 $\Omega$ and therefore $u$ has one sign in $\Omega$.  The simplicity of  
$\lambda_1(\Omega)$  follows repeating the standard arguments for the classical Laplace operator.
\end{proof}
\begin{rem}
\label{car}
Since eigenfunctions corresponding to different  eigenvalues are mutually orthogonal
in $L^2(\Omega,\phi_N)$,
the previous result implies that any positive function $v \in H^1(\Omega,\phi_N)$, which solves problem \eqref{eq:2} for some $\lambda \in \mathbb R$,  is a first eigenfunction,
that is $\lambda=\lambda_1(\Omega)$.
\end{rem}
\section{The eigenvalue problem in one dimension}
In this Section we consider problem \eqref{eq:2} in one dimension, that is
\begin{equation}
  \label{uni1}
  \left\{
    \begin{array}{ll}
-w''+tw'= \lambda(\sigma) w &  \,\, t \in  I_{\sigma}\\[.2cm] 
      w'(\sigma) +\beta w(\sigma)=0,
    \end{array}
  \right.
\end{equation}
where $I_{\sigma}=(-\infty,\sigma)$.

By Theorem \ref{esi}  the first eigenvalue $\lambda_1(\sigma)$ is simple and the corresponding eigenfunctions $w(t)  $ are smooth and have one sign in $I_{\sigma}$. In this case, the variational characterization reads as
\begin{equation}
\label{aut}
\lambda_1(\sigma)
=
\min_{\substack{v \in H^1\left(I_{\sigma},\phi_1\right)\setminus \{0\}}}\displaystyle \frac{\displaystyle \int_{-\infty}^{\sigma} (v'(t))^2 e^{-\frac{t^2}{2}}\,dt 
+
 \beta (v(\sigma))^2 e^{-\frac{\sigma^2}{2}}}{\displaystyle \int_{-\infty}^{\sigma} (v'(t))^2 e^{-\frac{t^2}{2}}\,dt},
\end{equation}
By standard theory on hypergeometric functions (see for example \cite{Eh},\cite{tri}), the minimizers $w$ of $\lambda_1(\sigma) $ have the following Taylor expansion
\begin{equation}
\label{sol1}
w(t)=C \sum_{m=0}^{\infty}\binom{\lambda_1(\sigma)}{m} \displaystyle \frac{\left( t/\sqrt 2\right)^m}{\Gamma\left( \frac{1-\lambda_1(\sigma)+m}{2}\right)},
\end{equation}
where $C \in \mathbb R$ and $\Gamma$ is the usual Gamma function. In particular when $\lambda_1(\sigma)\in \mathbb N$, the right hand-side in \eqref{sol1} reduces to the following Hermite polynomial of degree $\lambda_1(\sigma)$  
\[
w(t)= C \text{exp}\left( \frac{t^2}{2}\right) \displaystyle \frac{d^{\lambda_1(\sigma)}}{dt^{\lambda_1(\sigma)}} \text{exp}\left(- \frac{t^2}{2}\right).
\]
Moreover it is possible to show (see, for instance, \cite{tri}, pp. 34-35 and pp. 72-76) that $w$ has the following asymptotic behavior 
\begin{equation}
\label{asint}
w\propto t^{\lambda_1(\sigma)} \left(1+O(t^{-2}) \right) \quad \text{for } t \to -\infty .
\end{equation}
Problem \eqref{uni1} is strictly related to  problem \eqref{eq:2} when $\Omega$ is an half-space.
 More precisely let $\sigma \in \mathbb R$  and $S_{\sigma}$ as in \eqref{ss},  and let us consider the following problem
\begin{equation}
  \label{semi}
  \left\{
    \begin{array}{ll}
      -\divergenza\left(\phi_N(x)\nabla u)\right)=\lambda_1(S_\sigma) \phi_N(x) u(x)
      &\text{in } 
      S_{\sigma},\\[.2cm] 
      \displaystyle\frac{\partial u}{\partial x_1} +\beta u=0
      &\text{on } \{x_1=\sigma\},
    \end{array}
  \right.
\end{equation}
where  $\lambda_1(S_\sigma)$ is the first eigenvalue given by
 \begin{equation}
 \label{var}
 \lambda_1(S_\sigma)=\min_{v\in H^1(S_\sigma,\phi_N)\setminus \{0\}}\displaystyle \frac{\dint_{S_{\sigma}}|\nabla v|^2 \phi_N\,dx +\beta \dint_{\substack{\{x_1=\sigma\}}}v^2\phi_N \,d\mathcal{H}^{N-1}(x)}{\dint_{S_{\sigma}}v^2\phi_N\,dx}
 \end{equation}
 
Theorem \ref{esi} ensures that $\lambda_1(S_\sigma)$ is simple and  it admits a positive corresponding eigenfunction $u(x)$. In what follows we observe that the eigenfunctions $u$ are determined by the ones of problem \eqref{uni1}.  
 \begin{theo}
 \label{es_uni}
Let $u$ be a  positive eigenfunction corresponding to $\lambda_1(\sigma)$.
 Then  there exists a function 
$w(t) \colon I_{\sigma}\to (0,+\infty) $ 
such that $u(x)=w(x_1)$. Moreover $w$ is strictly monotone decreasing in $I_{\sigma}$ and it solves \eqref{uni1}, with $\lambda_1(\sigma)=\lambda_1(S_{\sigma})$.
 \end{theo}
 \begin{proof}
  Let $w(t)$ be a positive eigenfunction of problem \eqref{uni1} corresponding to  $\lambda_1(\sigma)$. Then it solves 
 \begin{equation}
  \label{uni_dx}
  \left\{
    \begin{array}{ll}
      -\left(w' \phi_1(t)\right)'=\lambda_1(\sigma) \phi_1(t) w
      &  \,\, t \in I_{\sigma},\\[.3cm] 
      w'(\sigma) +\beta w(\sigma)=0.
    \end{array}
  \right.
\end{equation}
Since $w>0$ in $I_{\sigma}$, we clearly have  that
\begin{equation}
  \label{<}
\left(w' \phi_1(t)\right)' < 0, \,\,\,\, t \in I_{\sigma}.
\end{equation}
We claim that
 \begin{equation}
  \label{claim}
\lim_{t \to -\infty}w'(t)e^{-\frac{t^2}{2}}=0.
\end{equation}
Note that, by \eqref{<}, such a limit exists. Now assume, by absurd, that
 $$
\lim_{t \to -\infty}w'(t)e^{-\frac{t^2}{2}}=L \in (- \infty, + \infty ] \setminus \{0\}.
$$
This would imply that
$$
\int_{-\infty}^{\sigma} (w'(t))^2 e^{-\frac{t^2}{2}}\,dt  = + \infty ,
$$
an absurd, since $w \in H^{1}(I_{\sigma}, \phi_{1})$. The claim \eqref{claim} is therefore proved.

From \eqref{claim}  and \eqref{<} it follows that  $w'(t) < 0 $ in $I_{\sigma}$.
 Defining $u(x)=w(x_1)$, $x \in S_{\sigma}$, by Theorem \ref{car} and the simplicity of the first eigenvalue, we get that $\lambda_1(\sigma)= \lambda_1(S_{\sigma})$.

 This concludes the proof of the Theorem.   
 \end{proof}
Let $w(t)$ as in Theorem \ref{es_uni} and let us define the following function
 \begin{equation}
 \label{beta}
 \beta(t)=\displaystyle -\frac{w'(t)}{w(t)}, \,\,\,\,t \in I_{\sigma},
 \end{equation}
 then it holds
 \begin{prop}
 \label{mon_beta}
 The function $\beta(t)$ defined in \eqref{beta} is positive, strictly increasing in $I_{\sigma}$ and $\beta(\sigma)=\beta$.
 \end{prop}
 \begin{proof}
 Being w a solution to \eqref{uni1}, Theorem \ref{es_uni} implies that $\beta >0.$ Moreover, since $w'$ verifies in $I_{\sigma}$ the following equation
\begin{equation}
\label{w'}
-(w')''+t(w')'= \left(\lambda(\sigma)-1\right) w', 
\end{equation}
taking into account the asymptotic behavior given in \eqref{asint},  we have that  
  \begin{equation}
  \label{lim}
  \lim_{t \to -\infty}\beta(t)=0.
  \end{equation}
Moreover we have that
 
  \[
  \beta'(t)=\displaystyle \frac{-w''w+(w')^2}{w^2} \qquad t \in I_{\sigma}.
  \] 
Since $w''$ is an eigenfunction of the one-dimensional  Hermite problem corresponding to the eigenvalue 
  $\lambda(\sigma)-2$, again \eqref{lim} yields 
    \begin{equation}
  \label{lim2}
  \lim_{t \to -\infty}\beta'(t)=0.
  \end{equation}
  Moreover by \eqref{uni1} it holds
  \[
  \beta'(t)= \displaystyle \frac{-tww' +(w')^2+\lambda(\sigma) w^2}{w^2}= t\beta(t) +\lambda(\sigma)+ \beta^2(t), \qquad t \in I_{\sigma}.
  \]
  Hence denoted by $z=\beta'$ we have
  \[
  z' =tz+\beta(t) +2\beta(t) z> z(2\beta(t)+t),  
  \]
  with $\lim_{t\to -\infty } z=0$. Then  $z> 0$ that is, $\beta(t)$ is  strictly increasing and this completes the proof.
 \end{proof}
 \begin{rem}
 We observe that Theorem \ref{es_uni} implies that any positive eigenfunction of  problem \eqref{uni1} is log-concave and  the same clearly, also true for problem \eqref{semi}. Indeed if we consider $$f(t)=\log(w(t)), \quad t \in I_{\sigma}$$ then $$f'(t)=-\beta(t) \text{ and } f''(t)=-\beta'(t)<0.$$ 
 \end{rem}
 Let $\lambda_1(\sigma)$ be  as in \eqref{aut}.  Proposition \ref{mon_beta} implies the following monotonicity result for the first eigenvalue $\lambda_1(\sigma)$ of the half-spaces.
 \begin{prop}
Let $r,\,\sigma \in \mathbb R$ such that $r\le \sigma$. Then $\lambda_1(r)\ge\lambda_1(\sigma)$. 
 \end{prop}
 \begin{proof}
 Let $\beta(r)$ be the function defined in \eqref{beta} and $\beta$ the parameter which appears in the Robin boundary condition in \eqref{uni1}, we have
 \begin{gather}
   \begin{split}
\lambda_1(\sigma)&=\displaystyle \frac{\dint_{-\infty}^{r} (w'(t))^2 e^{-\frac{t^2}{2}}\,dt + \beta(r) (w(r))^2 e^{-\frac{r^2}{2}}}{\dint_{-\infty}^{r} (w'(t))^2 e^{-\frac{t^2}{2}}\,dt}\\&
=
\min_{v \in H^1\left(I_{\sigma},\phi_1(t)\right)\setminus \{0\}}\displaystyle \frac{\dint_{-\infty}^{r} (v'(t))^2 e^{-\frac{t^2}{2}}\,dt + \beta(r) (v(r))^2 e^{-\frac{r^2}{2}}}{\dint_{-\infty}^{r} (v'(t))^2 e^{-\frac{t^2}{2}}\,dt}\\&
\le
 \min_{v \in H^1\left(I_{\sigma},\phi_1(t)\right)\setminus \{0\}}\displaystyle \frac{\dint_{-\infty}^{r} (v'(t))^2 e^{-\frac{t^2}{2}}\,dt + \beta \,\,(v(r))^2 e^{-\frac{r^2}{2}}}{\dint_{-\infty}^{r} (v'(t))^2 e^{-\frac{t^2}{2}}\,dt}\\ &
=
\lambda_1(r),
\end{split}
\end{gather}
where the last inequality follows by Proposition \ref{mon_beta}.
 \end{proof}
\section{A representation formula for $\lambda_1(\Omega)$}
\label{rap}
Let $\Omega \in \mathcal G$, with $\mathcal G$ as in Definition \ref{omega} and let  $u$ be the first positive eigenfunction of \eqref{eq:2} such that $\| u\|_{L^2(\Omega,\phi_N)}=1$. From now on, for every $t\in[0,1]$, we will use the following notation
\begin{equation*}
	\begin{array}{l}
		U_{t} =\{x\in \Omega\colon u>t\},\\[.1cm]
		\partial U^{\text{int}}_{t} =\{x\in \Omega\colon u=t\},\\[.1cm]
		\partial U^{\text{ext}}_{t}=\{x\in \de\Omega\colon  u>t\}.
	\end{array}
\end{equation*}
Let $\psi\in L^2(\Omega,\phi_N)$ be a non-negative function and let us consider the following functional
\begin{equation}
\label{f}
\f_{\Omega}(U_{t},\psi)=\frac{1}{\gamma_N(U_{t})}\left(
-\int_{U_{t}} \psi^{2} \phi_N\,dx + \int_{ \partial U^{\text{int}}_{t}} 
\psi \phi_N\,d\mathcal H^{N-1}(x) +
\beta \int_{ \partial U^{\text{ext}}_{t}} \phi_N\,d\mathcal H^{N-1}(x)\right).
\end{equation} 
The following  level set representation formula for $\lambda_1(\Omega)$ holds.
\begin{theo}
\label{theoform}
Let $\Omega \in \mathcal G$, with $\mathcal G$ as in Definition \ref{omega}  and  let $u$ be the positive minimizer of \eqref{eq:2} such that $\| u\|_{L^2(\Omega,\phi_N)}=1$. Then,
for a.e. $t>0$, it holds true 
\begin{equation}
\label{tesitest}
\lambda_1(\Omega)=\f_{\Omega}\left(U_{t}, \bar \psi \right),
\end{equation}
where $\bar \psi = \displaystyle\frac{|\nabla u|}{ u}$ and $\f_{\Omega}$ is defined in \eqref{f}.
\end{theo}
\begin{proof}
Being $u$ smooth and positive in $\Omega$, we can divide both terms in the equation in \eqref{eq:2} by $ u$, and integrate on $U_{t}$. Then,  by the boundary condition we get
\begin{multline}
\label{formal}
\begin{split}
\lambda(\Omega) \gamma_N(U_{t}) &=\displaystyle\int_{U_t}\displaystyle\frac{-\divergenza\left(\phi_N(x)\nabla u)\right)}{u}\,dx=-\displaystyle\int_{\partial U_t}\frac{\partial u}{\partial \nu} \frac{ 1}{u}\phi_N(x)\, d\mathcal H^{N-1}(x)-\int_{U_t}\displaystyle\frac{|\nabla u|^2}{u^2} \phi_N(x) \,dx\\&=\displaystyle\int_{\partial U^{\text{int}}_{t}}\frac{ |\nabla u|}{u}\phi_N(x)\, d\mathcal H^{N-1}(x)+\beta \displaystyle\int_{\partial U^{\text{ext}}_{t}}\phi_N(x)\, d\mathcal H^{N-1}(x)-\int_{U_t}\displaystyle\frac{|\nabla u|^2}{u^2} \phi_N(x) \,dx	\\&=\displaystyle\f_{\Omega}\left(U_{t}, \bar \psi\right),
\end{split}
\end{multline}
and this concludes the proof.
\end{proof}
\begin{theo}
\label{teotesithm}
Let $\Omega \in \mathcal G$, with $\mathcal G$ as in Definition \ref{omega}  and let $\bar\psi$  as in Theorem \ref{theoform}.  Let $\psi\in  L^2(\Omega,\phi_N)$ be a nonnegative function such that $\psi\not\equiv \bar \psi$ and  let $\mathcal F_{\Omega}$ be as in \eqref{f}. Set
\[
	w(x):=\psi-\bar \psi,\qquad I(t):=				\int_{U_{t}} w\bar\psi\phi_N\,dx,
	\] 
	then $I\colon ]0,+\infty[\rightarrow \mathbb R$ is locally absolutely continuous and
	\begin{equation}
	\label{firstineq}
	\mathcal F_{\Omega}(U_{t},\psi)\le\lambda_{1}(\Omega)-\frac{1}{\gamma_N(U_{t})}\frac{1}{t} \Big(\frac{d}{dt}t^{2} I(t) \Big),
	\end{equation}
	for almost every $t>0$. 
\end{theo}
\begin{proof}
In order to prove \eqref{firstineq}, writing the representation formula \eqref{tesitest} in terms of $w$, it follows that, for a.e. $t>0$,
\begin{equation}
\label{secondineq}
\begin{split}
\mathcal F_{\Omega}(U_{t},\psi)&= \lambda_{1}(\Omega)+\frac{1}{\gamma_N(U_{t})}
\left(
	\int_{\partial U_t^{\text{int}}} w\phi_N\,d\mathcal H^{N-1}(x)-\int_{U_{t}}
	\Big(\psi^{2}-\bar \psi^2\Big)\phi_N\,dx
\right)\\
&\le 
\lambda_{1}(\Omega)+\frac{1}{\gamma_N(U_{t})}
\left(
	\int_{\partial U_t^{\text{int}}} w\phi_N\,d\mathcal H^{N-1}(x)-2\int_{U_{t}}
	 w\bar\psi\phi_N \,dx
\right)\\
&=
\lambda_{1}(\Omega)+\frac{1}{\gamma_N(U_{t})}
\left(
	\int_{\partial U_t^{\text{int}}} w\phi_N\,d\mathcal H^{N-1}(x)-2\, I(t)
\right)
\end{split}
\end{equation}
where the inequality in \eqref{secondineq} follows since $\psi,\bar \psi\ge 0$. Applying the coarea formula, it is possible to rewrite $I(t)$ as
\[
	I(t)= \int_{U_{t}}w\bar\psi\phi_N\,dx = 
	\int_{t}^{+\infty}\frac{1}{\tau}d\tau \int_{\partial U_{\tau}^{\text{int}}} w\,\phi_N \,d\mathcal H^{N-1}(x).
\]
This assures that $I(t)$ is locally absolutely continuous in $]0,+\infty[$ and, for almost every $t>0$ we have
\[
-\frac{d}{dt}\big( t^2I(t) \big)=t \left(\int_{\partial U_t^{\text{int}}}w\,\phi_N\,d\mathcal H^{N-1}(x) -2I(t)\right).
\] 
Substituting in \eqref{secondineq}, the inequality \eqref{firstineq} follows. 
\end{proof}
\begin{theo}
\label{teotesithm}
Let $\Omega \in \mathcal G$, with $\mathcal G$ as in Definition \ref{omega},  and let $\bar\psi$  be as in Theorem \ref{theoform}.  Let $\psi\in L^2(\Omega,\phi_N)$ be a nonnegative function such that $\psi\not\equiv \bar \psi$ and  let $\mathcal F_{\Omega}$ be as in \eqref{f}.
Then there exists a set $T\subset ]0,+\infty[$ with positive Lebesgue measure such that for every $t\in T$ it holds that
\begin{equation}
	\label{cv}
	\lambda_1(\Omega)\ge\mathcal F_{\Omega}(U_{t},\psi).
\end{equation}
\end{theo}
\begin{proof}
Let $u$ be the first positive  eigenfunction of  \eqref{eq:2} such that $\|u\|_{L^2(\Omega,\phi_N)}=1$. We have to discuss two cases. If $u \in L^{\infty}(\Omega)$ then the claim follows by repeating line by line the arguments in \cite{da06} and  \cite{bd10}.  Hence,
in the proof below, we will assume that $u$ is not bounded. Note that, by the asymptotic behaviour given 
in \eqref{asint}, this case occurs surely when $\Omega$ is any half-space. 
 In order to prove \eqref{cv}, let us proceed by contradiction. Then we assume that there 
exists a nonnegative function $\psi \in L^2(\Omega,\phi_N)$ such that the reverse inequality holds, that is
\begin{equation}
\label
{ass}
\lambda_1(\Omega)< \mathcal F_{\Omega}(U_{t},\psi),
\end{equation}
for almost all $t>0$. Let $\psi_n \in C^{\infty}_c(\Omega)$ be a  sequence of nonnegative  functions  such that $\psi_n \to \psi$ in  $L^2(\Omega,\phi_N)$. Then by  Fatou's lemma we have  
\begin{equation*}
\lambda_1(\Omega ) <\mathcal F_{\Omega}(U_{t},\psi)\le\liminf_{n\to \infty}\mathcal F_{\Omega}(U_{t},\psi_n).
\end{equation*}
Let $\psi_{n_k}$ be a subsequence such that 
\[
\lim_{k \to \infty} F_{\Omega}(U_{t},\psi_{n_k})=\liminf_{n\to \infty}\mathcal F_{\Omega}(U_{t},\psi_n).
\]
In order to simplify the notation, we will still denote by $\psi_n$ such a subsequence
the subsequence. Then by \eqref{ass}  and Theorem \ref{teotesithm}   we have
\begin{equation*}
\lambda_1(\Omega )< \lim_{n\to \infty}\mathcal F_{\Omega}(U_{t},\psi_n)\le \lambda_1(\Omega )-\lim_{n\to \infty} \frac{1}{\gamma_N(U_{t})}\frac{1}{t} \Big(\frac{d}{dt}t^{2} I_n(t) \Big),
\end{equation*}
where $$I_n(t)= \int_{U_{t}}w_n\bar\psi\phi_N\,dx 
$$
with $w_n=\psi_n-\bar \psi$.
Then for $n$ large enough  it has to hold
\begin{equation}
\label{der}
\frac{d}{dt}t^{2} I_n(t)<0,
\end{equation}
almost everywhere in $]0,+\infty[$. Since
\begin{multline}
\begin{split}
\lim_{t \to +\infty} t^2 |I_n(t)|&\le\lim_{t \to +\infty}t^2 \int_{U_{t}}|\psi_n-\bar \psi|\bar\psi\phi_N\,dx\le\lim_{t \to +\infty} \int_{U_{t}}u^2|\psi_n-\bar \psi|\bar\psi\phi_N\,dx \\&\le\lim_{t \to +\infty} \int_{U_{t} \cap \text{supp}\psi_n}\psi_n u|\nabla u|\phi_N\,dx +\int_{U_{t}}|\nabla u|^2\phi_N\,dx =0,
\end{split}
\end{multline}
and then for sufficiently large  $n$ we have
\begin{equation}
\label{1}
\lim_{t \to +\infty} t^2 I_n(t)=0.
\end{equation}
On the other hand  denoted by $m=\min_{\bar\Omega}u>0$, it holds
\[
|I_n(t)|\le \int_{\Omega}|\psi_n-\bar \psi|\bar\psi\phi_N\,dx \le\frac{\gamma_N(\Omega)^{ 1/ 2}}{m} \|\psi_n\|_{L^{\infty}(\Omega \cap \text{supp}(\psi_n))}\|\nabla u\|_{L^2(\Omega,\phi_N)}+\frac{1}{m^2}\|\nabla u\|^2_{L^2(\Omega,\phi_N)}.
\]
Then  for sufficiently large  $n$ we have
\begin{equation}
\label{2}
\lim_{t \to 0^+} t^2 I_n(t) =0
\end{equation}
Since,  by \eqref{der},  the function $t^2 I_n(t)$ is strictly decreasing, equations \eqref{1} and \eqref{2} give an absurd.  This concludes the proof.
\end{proof}
\section{Proof of the main result}
In this Section we  prove the Faber-Krahn inequality stated in Theorem \ref{main}.

\begin{proof}[Proof of Theorem \ref{main}]
We  first  construct a suitable test function defined in $\Omega$ for \eqref{cv}.
Let $v$  be a positive eigenfunction to the problem \eqref{semi} in $\Omega^{\diesis}$ and $w$ the solution to \eqref{uni1} such that $v(x)=w(x_1)$. Then  if we  consider the function 
\[
	\varphi^*(r)=-\displaystyle \frac{w'(r)}{w(r)}	
	\quad\text{with } r\in I_{\sigma^{\diesis}}.
\]
By Theorem \ref{theoform}, we have 
\[
\lambda_1(\Omega^{\diesis})=\mathcal F_{\Omega^{\diesis}}(S_{r},\varphi^*) \]
As before, let $u$ be the first positive eigenfunction of \eqref{eq:2} in 
$\Omega$ such that $\|u\|_{L^2(\Omega,\phi_N)}=1$. Using the same notation of the previous Section, for any $t>0$ we consider $\mathcal S_{r(t)}$, the half-space such that $\gamma_N( U_t)=\gamma_N(S_{r(t)})$. Then, if $x\in\Omega$ and 
$u(x)=t$, we define the following test function
\[
	\psi(x):=\beta(r(t))=\psi^*(r(t)),
\]
where $\beta(r(t))$ is defined in \eqref{beta}.

We claim that the following inequality holds\begin{equation}
	\label{compar}
\mathcal F_{\Omega^{\diesis}}(S_{r(t)},\psi^*) \le\mathcal F_{\Omega}(U_{t},\psi)
\end{equation}
for  every  $t>0$. 
Being 
$$\mathcal F_{\Omega^{\diesis}}(S_{r(t)},\psi^*)= \frac{1}{\gamma_N( S_{r(t)})}
	\left( -\int_{S_{r(t)}} \psi^2\phi_Ndx + 
	\int_{\{x_1=\sigma^{\diesis}\}} \psi \phi_N\,d\mathcal H^{N-1}(x)
	\right),
	$$
in order to prove inequality \eqref{compar}, being $\gamma_N(U_{t})=\gamma_N(S_{r(t)})$ we have to show that 
\begin{gather}
-\int_{S_{r(t)}}\psi ^{2}\phi _{N}dx+\int_{\{x_{1}=\sigma ^{^{\#}}\}}\psi
\phi _{N}\,d\mathcal{H}^{N-1}(x)\leq   \label{confr_f} \\
-\int_{U_{t}}\psi ^{2}\phi _{N}\,dx+\int_{\partial U_{t}^{\text{int}}}\psi
\phi _{N}\,d\mathcal{H}^{N-1}(x)+\beta \int_{\partial U_{t}^{\text{ext}
}}\phi _{N}\,d\mathcal{H}^{N-1}(x).  \notag
\end{gather}

Regarding the first integral in the left and right hand-sides in \eqref{confr_f}, by construction we observe that the functions $\psi$ and $\psi^*$  are equimeasurable, with respect to the 
Gaussian measure, and therefore 
\begin{equation}
\label{eqpr}
	\int_{U_{t}} \psi^{2}\phi_Ndx =\int_{S_{r(t)}} (\psi^*)^{2}\phi_Ndx.
\end{equation}
Moreover, by the  isoperimetric inequality \eqref{isop}, and being, by Lemma \ref{mon_beta}, $\beta_{r(t)}\le \beta$ for all $t>0$, we have that
\begin{multline}
\label{ineqpr}
	\int_{\{x_1=r(t)\}} \psi^* \phi_N\,d\mathcal H^{N-1}(x)= \beta_{r(t)} P_{\phi_N}(S_{r(t)})\le \\\le \beta_{r(t)}P_{\phi_N}(U_{t}) 
	 \le \int_{\partial U_t^{\text{int}}} \psi  \phi_N \, d\mathcal H^{N-1}(x) + \beta \int_{\partial U_t^{\text{ext}}} \phi_N \, d\mathcal H^{N-1}(x).
\end{multline}
Hence, combining \eqref{eqpr} and \eqref{ineqpr} we get \eqref{confr_f} and then \eqref{compar}. 
Then, by Theorems \ref{theoform}, \ref{teotesithm} and equation \eqref{compar} we get, for $t \in T \in ]0,+\infty[$ defined in Theorem  \ref{teotesithm}, it happens that
\[
\lambda_1(\Omega^{\diesis})=
\mathcal F_{\Omega^{\diesis}}(S_{r(t)},\psi^*) \
\le \mathcal F_{\Omega}(U_{t},\psi)
\le \lambda_1(\Omega)
\]
which gives the claim. 

In order to conclude the proof, we have to consider the equality case. Let us suppose that $\lambda_1(\Omega)=\lambda_1(\Omega^{\diesis})$. Then all the inequalities appearing in the first part of the proof become equalities. In particular $P_{\phi_N}(S_{r(t)})=P_{\phi_N}(U_{t})$. By Theorem \ref{isop} we have that $U_t$ are half-spaces for each $t$. Since the level sets of any function are always nested, we have that $u$ depends  on $x_1$ only and it is a decreasing function. Therefore $\Omega=\Omega^{\diesis}$ and $u(x_1)$ coincides with $w$ modulo a constant. 
\end{proof}

\section*{Acknowledgements}
This work has been partially supported by the PRIN project 2017JPCAPN (Italy) grant: 
\lq\lq Qualitative and quantitative aspects of nonlinear PDEs\rq\rq 
and by GNAMPA of INdAM.

\end{document}